\documentclass{amsart}

\usepackage{amsmath,xypic,leftidx}
\usepackage{xspace}
\usepackage[psamsfonts]{amssymb}
\usepackage[latin1]{inputenc}
\usepackage{graphicx,color}

\usepackage{hyperref}

\theoremstyle{remark}

\newtheorem*{examples}{\bf Examples}

\theoremstyle{plain}

\newtheorem*{theorem}{\bf Theorem}

\newtheorem{lemma}{\bf Lemma}

\def\C{{\mathbb C}}

\def\N{{\mathbb N}}
\def\D{{\mathbb D}}

\def\and{{\quad\text{and}\quad}}

\def\i{{\sqrt{-1}}}
\def\e{{\rm e}}
\def\ii{{\bf i}}
\def\eps{{\varepsilon}}
\def\frak{\mathfrak }
\def\o{{\rm o}}

\title{Quadratic polynomials, multipliers and equidistribution}
\begin{author}[X.~Buff]{Xavier Buff}
\email{xavier.buff$@$math.univ-toulouse.fr}\thanks{The research of the first author was supported by the IUF}
\address{ %
 Université Paul Sabatier\\
Institut de Math\'ematiques de Toulouse\\
 118, route de Narbonne \\  
 31062 Toulouse Cedex \\
  France }
\end{author}
\begin{author}[T.~Gauthier]{Thomas Gauthier}
\email{thomas.gauthier$@$u-picardie.fr}
\address{ %
 Université de Picardie Jules Verne\\
LAMFA\\
33 rue Saint-Leu\\
80039 Amiens Cedex 1\\
  France }
\end{author}

\begin{document}

\begin{abstract}
Given a sequence of complex numbers $\rho_n$, we study the asymptotic distribution of the sets of parameters $c\in \C$ such that the quadratic maps $z^2+c$ has a cycle of period $n$ and multiplier $\rho_n$. Assume $\frac{1}{n}\log |\rho_n|\to L$. If $L\leq \log 2$, they equidistribute on the boundary of the Mandelbrot set. If $L>\log2$ they equidistribute on the equipotential of the Mandelbrot set of level $2L-2\log2$. 
\end{abstract}

\maketitle

\section*{Introduction}
In this article, we study equidistribution questions in the parameters space of the family of quadratic polynomials 
\[f_c(z):=z^2+c, \ c\in\C.\] 
We denote by $K_c$ the filled-in Julia set of $f_c$ and by $J_c$ the Julia set:
\[K_c:=\left\{z\in \C~|~\bigl(f_c^{\circ n}(z)\bigr)_{n\in \N}\text{ is bounded}\right\}\quad\text{and}\quad 
J_c:=\partial K_c.\]
The Mandelbrot set $M$ is the set of parameters $c\in \C$ such that $0\in K_c$. 
\begin{figure}[htb]
\vskip.5cm
\centerline{
\scalebox{.5}{\includegraphics{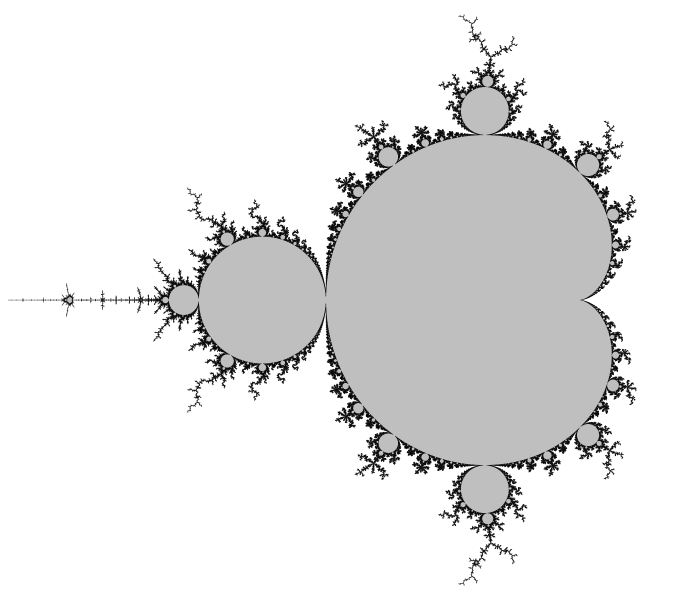}}
}
\caption{The Mandelbrot set $M$.\label{intro9}}
\end{figure}

The Green functions $g_c:\C\to [0,+\infty)$ and $g_M:\C\to [0,+\infty)$ are defined by 
\[g_c:=\lim_{n\to +\infty}\max\left(\frac{1}{2^n}  \log \bigl|f_c^{\circ n}(z)\bigr|,0\right)\quad\text{and}\quad g_M(c):=g_c(c).\]
The  bifurcation measure $\mu_{\rm bif}$ is defined by 
\[\mu_{\rm bif}:=\Delta g_M.\]
The support of the measure $\mu_{\rm bif}$ is the boundary of the Mandelbrot set $M$. 

~

Let $\rho$ be a complex number with $|\rho|\leq 1$. For $n\geq 1$, denote by $X_n$ the set of parameters $c\in \C$ such that the quadratic polynomial $f_c$ has a cycle of period $n$ with multiplier $\rho$. Let $\nu_n$ be the probability measure
\[\nu_n:=\frac{1}{{\rm card}(X_n)} \sum_{x\in X_n} \delta_x.\]
Bassanelli and Berteloot \cite{bbnagoya} proved that as $n\to +\infty$, the sequence of measures $\nu_n$ converges to the bifurcation measure $\mu_{\rm bif}$ (in fact, they prove a more general result, valid in families of polynomials of arbitrary degree). 
In this note, we prove that the result holds without the assumption that $|\rho|\leq 1$ and we study cases where the multiplier is not fixed. 

~

First, the set of parameters we will consider may fail to equidistribute on the boundary of the Mandelbrot set. They may equidistribute on an equipotential, i.e. a level curve of the function $g_M$. For $\eta\geq 0$, we let $\mu_\eta$ be the probability measure defined by 
\[\mu_\eta:= \Delta \max(g_M,\eta).\]
For $\eta>0$, the support of $\mu_\eta$ is exactly the equipotential $\{c\in \C~|~g_M(c)=\eta\}$. 

~

Second, we are interested in parameters $c\in \C$ for which $f_c$ has a cycle of period $n$ with multiplier $\rho$. When $|\rho|>1$, we may have to count such parameters $c$ with multiplicities. We proceed as follows. 
If $f_c$ has no parabolic cycle, we set: 
\[R_n(c,\rho):=\prod_{C\text{ cycle of }f_c\text{ of period }n} \bigl(\rho-\rho(C)\bigr)\]
where $\rho(C)$ is the multiplier of $C$ as a cycle of $f_c$. 
According to Bassanelli and Berteloot \cite{bbnagoya}, this defines a polynomial $R_n\in \C[c,\rho]$ which is therefore defined even when $f_c$ has parabolic cycles.

\begin{examples}
The fixed point of $f_c$ are the roots of $z^2-z +c$ and the multiplier of $f_c$ at a fixed point $z$ is $\rho=2z$. 
The resultant of $z^2-2z+c$ and $\rho-2z$ as polynomials in $z$ is $\rho^2-2\rho+4c$: 
\[R_1(c,\rho)= \rho^2-2\rho+4c.\]
The periodic points of period $2$ are the roots of  $z^2+z+c+1$ and the multiplier at a periodic point z is $\rho=4z(z^2+c)$. 
The resultant of $z^2+z+c+1$ and $\rho-4z(z^2+c)$ as polynomials in $z$ is $(\rho+4c+4)^2$:
\[R_2(c,\rho) = \rho+4+4c.\]
Similarly, one computes
\[R_3(c,\rho)=\rho^2-16\rho+64-8\rho c+64c+128c^2+64c^3.\]
\end{examples}

The degree $d_n$ of $R_n(c,\rho)$ as a polynomial of the variable $c$ does not depend on $\rho$.
It may be defined recursively by 
\[d_1=1\quad\text{and}\quad d_n= 2^{n-1}-\sum_{m \text{ divides } n \atop m\neq n} d_m.\]
As $n\to +\infty$ we have $d_n\sim 2^{n-1}$. 
For $n\geq 1$ and $\rho\in \C$, we set 
\[u_{n,\rho}:= \frac{1}{d_n} \log \bigl|R_n(c,\rho)\bigr|\quad\text{and}\quad \nu_{n,\rho}:=\Delta u_{n,\rho}.\]
The measure $\nu_{n,\rho}$ is a probability measure and its support is the set of parameters $c$ such that $f_c$ has a cycle of period $n$ and multiplier $\rho$ (except when $\rho=1$, in which case the support also contains parameters of period $k$ dividing $n$ whose multiplier is a $n/k$-th root of unity). 

~

Our result is the following. 
\begin{theorem}
Let $(\rho_n)_{n\geq 1}$ be a sequence of complex numbers such that 
\[\lim_{n\to +\infty} \frac{1}{n}\log |\rho_n| =L\in [-\infty,+\infty).\]
As $n\to+\infty$, the sequence $(\nu_{n,\rho_n})_{n\geq1}$ converges to $\mu_\eta$ with $\eta:=\max(0,2L-2\log 2)$. 
\end{theorem}

The note is organised as follows. In Section 1, we establish the Theorem relying on two Lemmas. The first one, which is proved in Section 2, is concerned with the asymptotic behavior of the sequence $(u_{n,\rho_n})$ outside of the compact set $\{c\in\C \ | \	 g_M(c)\leq\eta\}$. The second one is a direct application of a general comparison lemma for subharmonic functions on $\C$ which we establish in Section 3.

\section{Strategy of the proof of the Theorem}

Let us set 
\[v:=g_M+2\log 2:\C\to [0,+\infty).\]
Let $(\rho_n)_{n\geq 1}$ be a sequence of complex numbers such that 
\[\lim_{n\to +\infty} \frac{1}{n}\log |\rho_n| =L\in [-\infty,+\infty)\]
and set 
\[\eta:=\max(0,2L-2\log 2).\]
Our proof is based on the following lemmas. 

\begin{lemma}\label{lemma:limitu}
If $c\in\C-M$ satsifies $g_M(c)>\eta$, then 
\[\lim_{n\to +\infty} u_{n,\rho_n}(c) = v(c).\]
\end{lemma}

\begin{lemma}\label{lemma:comparisonM}
Any subharmonic function $u:\C\to [-\infty,+\infty)$ which coincides with $v$ outside $M$ coincides with $v$ everywhere. 
\end{lemma}

The proof is then completed as follows.  Extracting a subsequence if necessary, we may assume that the sequence of probability measures $\nu_{n_k,\rho_{n_k}}$ converges to some limit $\nu$. According to Lemma \ref{lemma:limitu}, the sequence $u_{n_k,\rho_{n_k}}$ then converges in $L^1_{\rm loc}$ to a limit $u$ which satisfies
\[\Delta u = \nu \quad\text{and}\quad u(c)=v(c)\text{ if }g_M(c)>\eta.\]
In addition, for every $c\in \C$, we have
\[u(c)\geq \limsup_{n\to +\infty} u_{n_k,\rho_{n_k}}(c)\]
with equality outside a polar set. According to Lemma \ref{lemma:limitu}, the subharmonic functions $u$ and $v$ coincide in the region $\bigl\{c\in \C~|~g_M(c)>\eta\bigr\}$. If $L\leq \log 2$, i.e. if $\eta=0$, Lemma \ref{lemma:comparisonM} implies that $u=v$ and so
\[\nu =\Delta u = \Delta (g_M+2\log 2) = \Delta g_M = \mu_{\rm bif}.\]

So, let us consider the case $L>\log 2$, or equivalently $\eta>0$. Note that  for $n\geq 2$ the cycles of $f_0(z)= z^2$ of period $n$ all have multiplier $2^n$. There are $k_n/n$ such points, where $k_n$ is the number of periodic points of period $n$ which may be defined recursively by 
\[k_1:=1\quad\text{and}\quad k_n = 2^n - \sum_{m\text{ divides }n\atop m\neq n} k_m.\]
As $n\to +\infty$, we have $k_n\sim 2^n$. Since $L>\log 2$, $2^n=\o(\rho_n)$ and so: 
\[u_{n,\rho_n}(0) = \frac{k_n/n}{d_n} \log|2^n-\rho_n| \sim \frac{2^n/n}{2^{n-1}} \log |\rho_n|\to 2L.\]
As a consequence, $u(0)\geq 2L$. According to  \cite[Theorem 3.8.3]{ransford},
\[\limsup_{g_M(c)\to \eta\atop g_M(c)<\eta} u(c) = \max_{g_M(c)=\eta} u(c) = \limsup_{g_M(c)\to \eta\atop g_M(c)>\eta} u(c) = 2L.\]
The Maximum Principle
for subharmonic functions implies that $u(c)=2L$ as soon as $g_M(c)\leq \eta$. 
So, 
\[u=\max(g_M+2\log 2,2L) = \max(g_M,\eta)+2\log 2\quad\text{and}\quad \nu = \mu_\eta.\]

\section{Outside the Mandelbrot set}

We now study the behaviour of the multipliers of $f_c$ when $c$ is not in $M$. Our aim in the present section is to prove Lemma \ref{lemma:limitu}.

First, when $c\in \C-M$, the Julia set $J_c$ of $f_c$ is a Cantor set and $f_c:J_c\to J_c$ is conjugate to the shift on $2$ symbols. 
In addition, when $c$ varies outside $M$, the Julia set moves locally holomorphically. However, it is not quite true that the holomorphic motion is parameterized by $\C-M$: if we start with $c\in (1/4,+\infty)$ and follow the Julia set as $c$ turns around the Mandelbrot set, every point comes back to its complex conjugate. After $2$ turns, the points come back to their initial location. See  \cite{bdk} for details. To avoid the monodromy problems, it is more convenient to pass to a cover of degree $2$ of $\C-M$. 

~

Let $\phi_M:\C-M\to \C-\overline \D$ be the conformal representation which sends $(1/4,+\infty)$ to $(1,+\infty)$ and for $\lambda\in \C-\overline \D$, set 
\[{\mathfrak c}(\lambda):= \phi_M^{-1}(\lambda^2).\]
Then, the holomorphic map ${\frak c}:\C-\overline\D\to \C-M$ is a covering map of degree $2$. 

Set $I:=\{0,1\}^\N$ and let $\sigma:I\to I$ be the shift. A point $\ii:=(i_0,i_1,i_2,\ldots )\in I$ is called an itinerary. There is a map $\psi:(\C-\overline \D)\times I\to \C$ such that 
\begin{itemize}
\item for all $\lambda\in \C-\D$, the map $\psi_\lambda:\ii\mapsto \psi(\lambda,\ii)$ is a bijection between $I$ and $J_{{\frak c}(\lambda)}$ conjugating $\sigma$ to $f_{{\frak c}(\lambda)}$: 
\[\psi_\lambda\circ \sigma = f_{{\frak c}(\lambda)}\circ \psi_\lambda\]
\item for all $\ii\in I$, the map $\psi_\ii:\lambda\mapsto \psi(\lambda,\ii)$ is holomorphic.
\end{itemize}
We may choose the map $\psi$ so that for $\lambda\in (1/4,+\infty)$, the map $\psi_\lambda$ sends itineraries for which $i_0=0$ in the upper half-plane and itineraries for which $i_0=1$ in the lower half-plane. Then, 
\[\psi_\ii(\lambda)\underset{\lambda\to +\infty}\sim \i\cdot \lambda\text{ if }i_0=0\quad\text{and}\quad
\psi_\ii(\lambda)\underset{\lambda\to +\infty}\sim -\i\cdot \lambda\text{ if }i_0=1.\]

Next, if $\ii$ is periodic of period $n$ for $\sigma$, let $\rho_\ii(\lambda)$ be the multiplier of $\psi_\ii(\lambda)$ as a fixed point of $f_{{\frak c}(\lambda)}^{\circ n}$. Note that 
\[\rho_\ii(\lambda):=\prod_{k=0}^{n-1} 2\cdot \psi_{\sigma^{\circ k}(\ii)}(\lambda)\underset{\lambda\to +\infty}\sim \eps_\ii\cdot (2\lambda)^n\quad \text{with}\quad \eps_\ii=\pm \cdot(\i)^n.\]
Since $\rho_\ii$ has local degree $n$ at $\infty$, we may write $\rho_\ii = \eps_\ii\cdot \sigma_\ii^n$ for some holomorphic map $\sigma_\ii:\C- \overline \D\to \C- \overline \D$ which satisfies
\[\sigma_\ii(\lambda)\underset{\lambda\to +\infty}\sim 2\lambda\quad\text{and}\quad 
\sigma_\ii(\omega \lambda) = \omega \cdot \sigma_\ii(\lambda)\text{ if }\omega^n=1.\]
The maps $\sigma_\ii$ take their values in $\C-\overline \D$ and so, form a normal family. 

Let $(\ii_n)$ be a sequence of periodic itineraries of period $n$. Let $\sigma: \C- \overline \D\to \C- \overline \D$ be a limit value of the sequence $(\sigma_{\ii_n}: \C- \overline \D\to \C- \overline \D)$. 
Then, $\sigma$ is tangent to $\lambda\mapsto 2\lambda$ at infinity and commutes with rotations. Thus, $\sigma(\lambda)= 2\lambda$. 
Now, if $c\in \C-M$, then $c={\frak c}(\lambda)$ with $\log |\lambda| = \frac{1}{2}g_M(c)$. As $n\to +\infty$, 
\[\log \bigl|\sigma_{\ii_n}(\lambda)\bigr|-\frac{1}{n}\log |\rho_n|\underset{n\to +\infty}\longrightarrow \log |2\lambda|-L=\frac{g_M(c)-\eta}{2}.\]
So, if $g_M(c)>\eta$, then there is an $\eps>0$ such that for any periodic itinerary $\ii$ of high enough period $n$, we have 
\[\left|\frac{\rho_n}{\rho_\ii(\lambda)} \right|= \left|\frac{\rho_n}{\sigma_\ii(\lambda)^n}\right|\leq \e^{-n\eps} .\]
In that case 
\begin{eqnarray*}
u_{n,\rho_n}(c)  &=& \frac{1}{d_n}\sum_{\ii\text{ periodic itinerary}\atop\text{ of period }n} \frac{1}{n}\log \bigl|\rho_\ii(\lambda)-\rho_n\bigr|\\
&=& \frac{1}{d_n}\sum_{\ii\text{ periodic itinerary}\atop\text{ of period }n} \log \bigl|\sigma_\ii(\lambda)\bigr| + \frac{1}{n}\log \left|1-\frac{\rho_n}{\rho_\ii(\lambda)} \right|\\
&\underset{n\to +\infty}=& \frac{k_n}{d_n} \log|2\lambda| + \o\left(\frac{k_n}{d_n}\right) \sim  2\log |2\lambda| = g_M(c) + 2\log 2,
\end{eqnarray*}
which ends the proof of Lemma \ref{lemma:limitu}.

\section{Comparison of subharmonic functions}

The proof of Lemma \ref{lemma:comparisonM} relies on the following more general result. 

\begin{lemma}\label{lemma:comparisonK}
Let $K\subset \C$ be a compact set such that $\C-K$ is connected. Let $v$ be a subharmonic function on $\C$ such that $\Delta v$ is supported on $\partial K$ and does not charge the boundary of the connected components of the interior of $K$. Then, any subharmonic function 
$u$ on $\C$ which coincides with $v$ outside $K$ coincides with $v$ everywhere.
\end{lemma}

\begin{proof}
According to  \cite[Theorem 3.8.3]{ransford}, for all $\zeta \in \partial K$, 
\begin{equation}
\label{eq:usc1}
u(\zeta) = \limsup_{z\in \C-K\atop z \to \zeta} u(z) =  \limsup_{z\in \C-K\atop z \to \zeta} v(z) =v(\zeta).
\end{equation}
So, $u=v$ on $\partial K$. 
The same theorem shows that, if $\zeta$ belongs to the boundary of a connected component $U$ of the interior of $K$, then 
\begin{equation}
\label{eq:usc2}
\limsup_{z\in \overline U\atop z \to \zeta} u(z) = u(\zeta)=v(\zeta) =  \limsup_{z\in \overline U\atop z \to \zeta} v(z).
\end{equation}

First, consider the function $w_1$ defined by 
\[w_1(z) = \begin{cases}\max(u,v) &\text{on } U\\ 
v &\text{on }\C-U.
\end{cases}
\]
According to Equations (\ref{eq:usc1}) and  (\ref{eq:usc2}), $w_1$ is  upper-semicontinuous. It is subharmonic on $U$ and $\C- \overline U$. At any point $\zeta\in \partial U$, it satisfies the 
local submean inequality since $w_1(\zeta)=v(\zeta)$ and $v\leq w_1$ on $\C$. Thus, $w_1$ is globally subharmonic on $\C$ and coincides with $v$ outside $U$.  
Consider a smooth test function $\chi:\C\to [0,1]$ which vanishes near $\infty$ and is constant equal to $1$ near $K$. 
On the one hand, 
\[\Delta v(\C) = \int_\C \chi\cdot \Delta v = \int_\C \Delta \chi \cdot v = \int_\C \Delta \chi \cdot w_1 = \int_\C \chi \cdot \Delta w_1 = \Delta w_1(\C).\]
On the other hand, $\Delta v = \Delta w_1$ outside $\overline U$. 
Therefore, 
\[\Delta w_1(\overline U) = \Delta w_1(\C)-\Delta w_1(\C-\overline U) =  \Delta v(\C)-\Delta v(\C-\overline U) =\Delta v(\overline U)=0.\]
So, $\Delta v = \Delta w_1$ on $\C$.
The difference $v-w_1$ is harmonic and vanishes outside $U$. Therefore $w_1=v$ on $\C$ and so, $\max(u,v)=v$ on $U$. 
It follows that $u\leq v$ on $U$. Since this holds for any connected component $U$ of the interior of $K$, we have 
$u\leq v$ on $\C$. 

Next, consider the function  $w_2$ defined by 
\[w_2(z) = \begin{cases} u &\text{on } U\\ 
v &\text{on }\C-U.
\end{cases}
\]
As previously, $w_2$ is upper-semicontinuous and subharmonic on $U$ and $\C- \overline U$. At any point $\zeta\in \partial U$, it satisfies the 
local submean inequality since $w_2(\zeta)=u(\zeta)$ and $u\leq w_2$ on $\C$. As previously, $\Delta v=\Delta w_2$ outside $\overline U$ and vanishes on $\overline U$, so that $\Delta v=\Delta w_2$ on $\C$. The function $v-w_2$ is harmonic on $\C$ and vanishes outside $\overline U$. So, $w_2=v$ on $\C$. In particular $u=v$ on $U$. 

Since this is valid for any connected component $U$ of the interior of $K$, we have $u=v$ on $\C$ as required. 
\end{proof}

Let us remark that the proof can be simplified if $v$ is continuous, which is the case in our situation.

%According to Douady \cite{douady}, the bifurcation measure of the boundary of the main cardioid of the Mandelbrot set is zero. Any other component of the interior of the Mandelbrot set is contained in a copy of the Mandelbrot set centered at a parameter $c_0$ such that $0$ is periodic of period $p\geq 2$ for $f_{c_0}$. If an external ray of angle $t$ of the Mandelbrot set accumulates on such a copy, the external ray of $K_{c_0}$ of angle $t$ accumulates on the Julia set of a quadratic-like restriction $f_{c_0}^{\circ p}:U\to V$. The set of such angles is a closed subset $T$ of $\R/\Z$ such that  each angle in $T$ has exactly $2$ preimages within $T$ by the map $t\mapsto 2^p \cdot t$. It follows that the Hausdorff dimension of $T$ is $\log (2)/\log(2^p) = 1/p$. So, the bifurcation measure of the copies of $M$ of period $p\geq 2$ is zero. 

\begin{proof}[Proof of Lemma \ref{lemma:comparisonM}]
According to Zakeri \cite{zakeri}, there is a set of parameters $c$ which is of full measure for $\mu_{\rm bif}$ such that 
\begin{itemize}
\item $J_c$ is locally connected and full, in particular $c$ is not in the boundary of a hyperbolic component of $M$ and
\item the orbit of $c$ is dense in $J_c$, in particular $c$ is not renormalizable and so, not in the boundary of a queer component of $M$. 
\end{itemize}
As a consequence, $\mu_{\rm bif}$ does not charge the boundary of connected components of the interior of $M$. 
So, we may apply Lemma \ref{lemma:comparisonK} with $K:=M$ and $v:=g_M+2\log 2$. This yields Lemma  \ref{lemma:comparisonM}.
\end{proof}

\end{document}